\documentclass{amsart}
\usepackage{latexsym,amscd,amssymb,url}
\usepackage{extarrows}
\usepackage{verbatim}
\usepackage{bbm}
\usepackage{dsfont}
\usepackage{tikz-cd}
\usepackage[all,cmtip]{xy}  
\usepackage{graphicx}
\usepackage{xcolor}
\usepackage{enumitem} 
\definecolor{dblue}{rgb}{0,0,.6}
\usepackage[colorlinks=true, linkcolor=dblue, citecolor=dblue, filecolor = dblue, menucolor = dblue, urlcolor = dblue]{hyperref}

\numberwithin{equation}{section}

\newtheorem{theorem}{Theorem}[section]

\theoremstyle{plain}

\newtheorem{corollary}[theorem]{Corollary}

\newtheorem{lemma}[theorem]{Lemma}

\newtheorem{proposition}[theorem]{Proposition}
\newtheorem{remark}[theorem]{Remark}

\newcommand{\del}{\partial}
\newcommand{\Z}{\mathbb Z}
\newcommand{\Q}{\mathbb Q}

\newcommand{\C}{\mathbb C}

\newcommand{\D}{\operatorname{D}}

\newcommand{\im}{\operatorname{im}}

\newcommand{\Hom}{\operatorname{Hom}}

\newcommand{\Spec}{\operatorname{Spec}}

\newcommand{\codim}{\operatorname{codim}}

\newcommand{\CH}{\operatorname{CH}}
\newcommand{\supp}{\operatorname{supp}}
 
\newcommand{\red}{\operatorname{red}}

\newcommand{\cl}{\operatorname{cl}}

 \newcommand{\sm}{\operatorname{sm}}

  \newcommand{\coker}{\operatorname{coker}}

\newcommand{\tors}{\operatorname{tors}}

\newcommand{\alg}{\operatorname{alg}}

\newcommand{\Sh}{\operatorname{Shv}}
\newcommand{\et}{\text{\'et}}
\newcommand{\proet}{\text{pro\'et}}

\newcommand{\colim}{\operatorname{colim}}

\newcommand{\DM}{\operatorname{DM}}
\newcommand{\h}{\operatorname{h}}
\newcommand{\BM}{\operatorname{BM}}
\newcommand{\BMM}{\operatorname{BM,M}}

\newcommand{\dashedlongrightarrow}{\xymatrix@1@=15pt{\ar@{-->}[r]&}}
\renewcommand{\longrightarrow}{\xymatrix@1@=15pt{\ar[r]&}}
\renewcommand{\mapsto}{\xymatrix@1@=15pt{\ar@{|->}[r]&}}
\renewcommand{\twoheadrightarrow}{\xymatrix@1@=15pt{\ar@{->>}[r]&}}
\newcommand{\hooklongrightarrow}{\xymatrix@1@=15pt{\ar@{^(->}[r]&}}
\newcommand{\congpf}{\xymatrix@1@=15pt{\ar[r]^-\sim&}}
\renewcommand{\cong}{\simeq}

\begin{document}

\title[Two Cycle Class maps on Torsion Cycles]{Two Cycle Class maps on Torsion Cycles}
\author{Theodosis Alexandrou} 
\address{Institute of Algebraic Geometry, Leibniz University Hannover, Welfengarten 1, 30167 Hannover , Germany.}
\email{alexandrou@math.uni-hannover.de}

\date{\today} 
%
% 14J70 Hypersurfaces
% 14J45 Fano varieties
% 14J10 Families, moduli, classification: algebraic theory
% 	14J35   	$4$-folds
% 14M20   	Rational and unirational varieties
% 	14M22   	Rationally connected varieties
% 14D06   	Fibrations, degenerations
%  	14E08   	Rationality questions
%  14C15   	(Equivariant) Chow groups and rings; motives
% 	14C25   	Algebraic cycles
%  	14M10   	Complete intersections
% 	14C30   	Transcendental methods, Hodge theory [See also 14D07, 32G20, 32J25, 32S35], Hodge conjecture
%	14F20   	Étale and other Grothendieck topologies and (co)homologies


 \begin{abstract} We compare two cycle class maps on torsion cycles and show that they agree up to a minus sign. The first one goes back to Bloch (1979), with recent generalizations to non-closed fields. The second is the \'etale motivic cycle class map $\alpha^{i}_{X}\colon \CH^{i}(X)_{\Z_{\ell}}\to H^{2i}_{L}(X,\Z_{\ell}(i))$ restricted to torsion cycles.\end{abstract}

\maketitle

\section{Introduction}
 Let $X$ be a smooth variety over a field $k$ and let $\ell$ be a prime invertible in $k$. One approach to study the torsion in the Chow group $\CH^{i}(X)$ of codimension $i$-cycles on $X$ is via cycle class maps. If the field $k$ is algebraically closed and $X$ is projective, a notable example of such a map is that of Bloch's (see \cite{bloch-compositio}) \[ \lambda^{i}_{X}\colon \CH^{i}(X)[\ell^{\infty}]\longrightarrow H^{2i-1}(X_{\et},\Q_{\ell}/\Z_{\ell}(i)),\] with values in \'etale cohomology and where $\CH^{i}(X)[\ell^{\infty}]$ denotes the $\ell$-power torsion subgroup of $\CH^{i}(X)$. This map becomes particularly interesting, when restricted to $i\leq 2$ or $i=\dim X$. The reason for this is that $\lambda^{i}_{X}$ is known to be injective in these cases. In particular, $\lambda^{1}_{X}$ is simply the isomorphism that arises from the Kummer sequence. The injectivity of $\lambda^{2}_{X}$ is a result of Bloch and Merkurjev-Suslin (see \cite[$\S$18]{MS}), whereas the injectivity of $\lambda^{\dim X}_{X}$, recovers the celebrated theorem of Roitman \cite{roitman}. For the remaining cases, i.e. $3\leq i\leq \dim X-1$, the map $\lambda^{i}_{X}$ in general fails to be injective (see \cite{schoen-product, totaro-JAMS, SV, RS, totaro-annals, Sch-griffiths, Ale-griffiths}).\par Some efforts have been made to generalise Roitman's theorem to non-proper and even singular varieties (see \cite{levine-AJM, bs, krishna-srinivas, geisser}), as well as to extend it over finite fields \cite{CTSS,KS}. In pursuit of generalizing Roitman's theorem over arbitrary fields, the authors of \cite{alexandrou-schreieder}, following \cite{bloch-compositio, Sch-refined}, constructed a homomorphism 
 \begin{align} \label{def:blochs-map} \lambda^{i}_{X}\colon \CH^{i}(X)[\ell^{\infty}]\longrightarrow H^{2i-1}(X_{\et},\Q_{\ell}/\Z_{\ell}(i))/M^{2i-1}(X),
 \end{align} that extends Bloch's definition to non-closed fields and where the group $M^{2i-1}(X)$ is explicitly defined as follows: Fix a finitely generated subfield $k_{0}$ of $k$, such that $X$ admits a model $X_{0}$ over $k_{0}$, i.e. $X=X_{0}\times_{k_{0}}k$ and let \begin{align}\label{def:M^2i-1} M^{2i-1}(X)=\im\left( \underset{k_{0}\subset E\subset k}{\colim} N^{i-1} H^{2i-1}_{cont}({X_0}_{E},\Q_\ell(i))\longrightarrow  H^{2i-1}_{cont}(X,\Q_\ell/\Z_\ell(i))\right),\end{align} where the colimit runs over the finitely generated subfields $E$ of $k$ containing $k_{0}$ and where $N^{\ast}$ denotes the coniveau filtration (see \eqref{eq:coniveau-filtration}). The latter map is shown to be injective for all $i\leq2$ and not in general if $3\leq i \leq\dim X$ (see \cite[Thm. 1.1 and Thm. 1.2]{alexandrou-schreieder}). Especially, Roitmans theorem \cite{roitman} fails over arbitrary fields.\par In addition to Bloch's map, another well-known example is the \'etale motivic cycle class map \begin{align}\label{def:motivic-cycle-map}\alpha^{i}_{X}\colon \CH^{i}(X)_{\Z_{\ell}}\longrightarrow H^{2i}_{L}(X,\Z_{\ell}(i)),\end{align} with values in Lichtenbaum (or \'etale motivic) cohomology. This is simply the comparison map $$H^{j}_{M}(X,\Z_{\ell}(i))\longrightarrow H^{j}_{L}(X,\Z_{\ell}(i))$$ in case $j=2i$, where the subscript $M$ denotes motivic cohomology \cite{v-s-f}. It is known that \eqref{def:motivic-cycle-map} is also injective if $i\leq2$ (see \cite[Prop. 2.5]{kahn}). Moreover, Kahn noticed that some of the examples from \cite{scavia-suzuki, alexandrou-schreieder} serve as counterexamples to the injectivity of $\alpha^{i}_{X}$, as well (see \cite[Sec. 7]{kahn23}).\par When restricted to torsion cycles \eqref{def:motivic-cycle-map} takes the form \begin{align}\label{def:etale-motivic-torsion-map}\alpha^{i}_{X,\tors}:\CH^{i}(X)[\ell^{\infty}]\longrightarrow H^{2i-1}(X_{\et},\Q_{\ell}/\Z_{\ell}(i))/I^{2i-1}(X),\end{align} where $I^{2i-1}(X):=\im(H^{2i-1}_{M}(X,\Q_{\ell}(i))\longrightarrow H^{2i-1}(X_{\et},\Q_{\ell}/\Z_{\ell}(i)))$ and where the isomorphism $$H^{2i}_{L}(X,\Z_{\ell}(i))_{\tors}\cong H^{2i-1}(X_{\et},\Q_{\ell}/\Z_{\ell}(i))/I^{2i-1}(X)$$ follows readily from the Bockstein sequence (see Prop. \ref{prop:twisted-borel-moore-axioms} \ref{P3} below). If $k$ is algebraically closed and $X$ is projective, then $I^{2i-1}(X)=0$ by weight reasons. In case $k=\C$, an inspection of the argument in \cite[Prop. 5.1.(c)]{RS16} gives that \eqref{def:etale-motivic-torsion-map} is Bloch's map up to a minus sign.\par In light of the properties that both \eqref{def:blochs-map} and \eqref{def:etale-motivic-torsion-map} share it is reasonable to wonder whether $I^{2i-1}(X)=M^{2i-1}(X)$ and specifically, if the two maps coincide. We dedicate this note to show the following. 
 \begin{theorem}\label{thm:lambda=alpha} Let $X$ be a smooth variety over a field $k$ and let $\ell$ be a prime invertible in $k$. Let $\lambda^{i}_{X}$ and $\alpha^{i}_{X,\tors}$ be the cycle class maps on $\ell$-power torsion cycles from \eqref{def:blochs-map} and \eqref{def:etale-motivic-torsion-map}, respectively. Then $I^{2i-1}(X)=M^{2i-1}(X)$ and $\lambda^{i}_{X}=-\alpha^{i}_{X,\tors}$ for all $1 \leq i\leq\dim X$.\end{theorem}
 The idea of the proof relies on the fact that one can use Borel-Moore motivic \'etale cohomology to reformulate $\alpha^{i}_{X,\tors}$ in such a way that the relation between the two maps becomes clear.
 \subsection{Notations} We adopt the following notations and conventions: $k$ is a field. An algebraic $k$-scheme $X$ is a seperated scheme of finite type over $k$. A variety $X$ is an integral algebraic $k$-scheme. Given an abelian group $A$, $A[m]:=\ker(A\overset{m}{\to}A)$ is the kernel of multiplication by $m$ and we let $A[\ell^{\infty}]=\bigcup _{r=1}^{\infty}A[\ell^{r}]$. If $f\colon A\to B$ is a homomorphism of abelian groups, we denote by abuse of notation $B/A:=\coker(f)$.
 \section{Twisted Borel-Moore \'Etale Motivic Cohomology}
 \subsection{Borel-Moore motivic \'etale cohomology}\label{sub:1} We gather some of the important properties of Borel-Moore motivic \'etale cohomology for later use. The results of this subsection are contained in \cite{v-s-f} and \cite{cd}.\par Let $R$ be a commutative ring. For a noetherian scheme $X$, we let $\DM_{\h}(X,R)$ be the category of $\h$-motives as defined in \cite[Def. 5.1.3]{cd}. By \cite[Thm 5.6.2]{cd}, the triangulated premotivic category $\DM_{\h}(-,R)$ satisfies the formalism of the Grothendieck 6 functors (see \cite[Def. A.1.10]{cd}) for noetherian schemes of finite dimension. %as well as the absolute purity property (see \cite[Def. A.2.9]{cd}) 
 Thus one can define the Borel-Moore motivic \'etale cohomology group of an algebraic scheme $X$ of dimension $d_{X}$ by the formula \begin{align}\label{def:Borel-Moore-Motivic}
 H^{i}_{\BMM}(X_{\et},R(j)):=\Hom_{\DM_{\h}(X,R)}(\mathbbm{1}_{X}(d_{X}-j)[2d_{X}-i],\pi^{!}_{X}(\mathbbm{1}_{\Spec(k)})),  
 \end{align}
 where $\mathbbm{1}_{X}$ denotes the identity object for $\otimes$ in $\DM_{\h}(X,R)$ and $\pi_{X}:X\to\Spec(k)$ is the structure morphism. We shall use the notation $R_{X}$ instead of $\mathbbm{1}_{X}$, whenever we want to give emphasis on the ring of coefficients $R$.\par Let $\ell$ be a prime invertible in the field $k$. The ring $R$ will always be taken from the set $\{\Z_{\ell},\Q_{\ell},\Z/{\ell^{r}}\}$. We also define \begin{align} \label{def:Q/Z-coefficients}
     H^{i}_{\BMM}(X_{\et},\Q_{\ell}/\Z_{\ell}(j)):=\underset{r\to\infty}{\colim}\ H^{i}_{\BMM}(X_{\et},\Z/\ell^{r}(j)).
 \end{align} 
 By \cite[Prop. 6.3.4 and Cor. 5.5.14]{cd}, one gets a canonical isomorphism \begin{align}\label{def:Q-coeeficients}
 H^{i}_{\BMM}(X_{\et},\Q_{\ell}(j))\cong H^{i}_{\BMM}(X_{\et},\Z_{\ell}(j))\otimes\Q_{\ell}. \end{align}
If $R=\Z/\ell^{r}$, then by \cite[Cor. 5.5.4]{cd}, we have an equivalence of premotivic triangulated categories over the category of finite dimensional noetherian schemes: \[ D(-_{\et},\Z/\ell^{r})\cong\DM_{\h}(-,\Z/\ell^{r}).\] We will identify \eqref{def:Borel-Moore-Motivic} with Borel-Moore \'etale cohomology in this case via the canonical isomorphism\begin{align}\label{etale=motivic} H^{i}_{\BMM}(X_{\et},\Z/\ell^{r}(j))\cong H^{i}_{\BM}(X_{\et},\Z/\ell^{r}(j)).\end{align}
\par Let $R\in\{\Z_{\ell},\Q_{\ell},\Z/{\ell^{r}}\}$. If $X$ is smooth and equi-dimensional of dimension $d_{X}$, then the natural isomorphism $\pi_{X}^{*}(d_{X})[2d_{X}]\cong\pi^{!}_{X}$, yields \begin{align}\label{etale-vs-lichtenbaum}
    H^{i}_{\BMM}(X_{\et},R(j))\cong H^{i}_{L}(X,R(j)).
\end{align}
 Here the subscript $L$ corresponds to Lichtenbaum (or \'etale motivic) cohomology (see \cite[Thm. 7.1.2]{cd} and \cite{voe}). If we restrict to $R=\Q_{\ell}$, then by \cite[Prop. 2.2.10]{cd}, the last isomorphism also gives \begin{align}\label{etale-vs-motivic}
H^{i}_{\BMM}(X_{\et},\Q_{\ell}(j))\cong H^{i}_{M}(X,\Q_{\ell}(j)),\end{align} where $H^{i}_{M}(-,R(j))$ denotes motivic cohomology (see \cite{v-s-f}).\par We close this subsection with two Lemmas that will be used frequently.
\begin{lemma}\label{lemma:topological-invariance} Let $X$ be an algebraic $k$-scheme and let $\ell$ be a prime invertible in $k$. Fix a perfect closure $k':=k^{\text{perf}}$ for the field $k$ and let $R\in\{\Z_{\ell},\Q_{\ell},\Z/{\ell^{r}}\}$. We have natural isomorphisms $$ H^{i}_{\BMM}(X_{\et},R(j))\cong H^{i}_{\BMM}({X_{k'}}_{\et},R(j))\ \text{and}\ H^{i}_{\BMM}(X_{\et},R(j))\cong H^{i}_{\BMM}({X_{\red}}_{\et},R(j)) $$ for all $i,j\in\Z$.\end{lemma}\begin{proof} Indeed, let $d=\dim X$ and consider the cartesian square \begin{center}
    \begin{tikzcd}
X':=X_{k'} \arrow[d, "p'"] \arrow[r, "\pi_{X'}"] & \Spec k' \arrow[d, "p"] \\
X \arrow[r, "\pi_{X}"]                           & \Spec k.              
\end{tikzcd}
\end{center} By adjunction, we clearly have $$H^{i}_{\BMM}(X_{\et},R(j))=\Hom_{\DM_{\h}(k,R)}({\pi_{X}}_{!}\mathbbm{1}_{X}(d-j)[2d-i],\mathbbm{1}_{\Spec(k)}).$$ Using the natural isomorphism $p^{\ast}(\pi_{X})_{!}\cong (\pi_{X'})_{!}p'^{\ast}$ (see \cite[Def. A.1.10 (4)]{cd}), we see that the pull-back $p^{\ast}\colon\DM_{\h}(k,R)\longrightarrow \DM_{\h}(k',R)$ induces a homomorphism $$
    p^{\ast}\colon H^{i}_{\BMM}(X_{\et},R(j))\longrightarrow H^{i}_{\BMM}(X'_{\et},R(j)).
$$
The first isomorphism follows as the pull-back functor $p^{\ast}$ is an equivalence by \cite[Prop. 6.3.16]{cd}. Since $\DM_{\h}(-,R)\cong\DM_{\h}(-_{\red},R)$, the second isomorphism is also clear.\end{proof}
\begin{lemma}\label{lemma:colimit} Let $X$ be an algebraic $k$-scheme and let $\ell$ be a prime invertible in $k$. Fix a field $k_{0}$ over which $X$ has a model $X_{0}$, i.e. $X=X_{0}\times_{k_{0}}k$ and let $R\in\{\Z_{\ell},\Q_{\ell},\Z/{\ell^{r}}\}$. For all $i,j\in \Z$, we have \[H^{i}_{\BMM}(X_{\et},R(j))=\underset{k_{0}\subset E\subset k}{\colim}H^{i}_{\BMM}({{X_{0}}_{E}}_{\et},R(j)),\]where the colimit runs over all finitely generated field extensions $k_{0}\subset E$ with $E\subset k$.\end{lemma}\begin{proof} Following the proof of the preceding Lemma, we find that the groups $H^{i}_{\BMM}({{X_{0}}_{E}}_{\et},R(j))$ form a direct system over the set of all finitely generated field extensions $k_{0}\subset E$ with $E\subset k$. In particular, we get a natural map \[\underset{k_{0}\subset E\subset k}{\colim}H^{i}_{\BMM}({{X_{0}}_{E}}_{\et},R(j))\longrightarrow H^{i}_{\BMM}(X_{\et},R(j)),\] that we claim to be an isomorphism. This follows immediately from \cite[Prop. 6.3.7]{cd}.\end{proof}
 \subsection{Twisted Borel-Moore axioms}\label{sub:2} It is clear from the definition that the Borel-Moore motivic \'etale cohomology groups \eqref{def:Borel-Moore-Motivic} are contravariant functorial with respect to \'etale maps $j\colon U\to X$, such that $\dim U=\dim X$. Furthermore, we have: 
 \begin{proposition}\label{prop:twisted-borel-moore-axioms} Let $R\in\{\Z_{\ell},\Z/\ell^{r},\Q_{\ell}\}$. The family of contravariant functors \[X \mapsto H^{i}(X,R(j)),\ i,j\in\Z,\] where we write $H^{i}(X,R(j)):=H^{i}_{\BMM}(X_{\et},R(j))$ for simplicity, satisfies the following properties:\begin{enumerate}[label=$(\roman*)$]
\item \label{P1} For any proper morphism $f\colon X\to Y$ of algebraic $k$-schemes and of relative codimension $c:=\dim Y-\dim X$, there are functorial pushforward maps $f_{*}\colon H^{i-2c}(X,R(j-c))\to H^{i}(X,R(j))$ such that for any cartesian diagram \begin{center}
         \begin{tikzcd}
V \arrow[d, "f'"] \arrow[r, "j'"] & X \arrow[d, "f"] \\
U \arrow[r, "j"]                  & Y               
\end{tikzcd}
     \end{center} where the map $j\colon U\to Y$ is \'etale with $\dim U=\dim Y$ and $\dim V=\dim X$, we have $j^{\ast}f_{\ast}=f'_{\ast}j'^{\ast}$.
    \item\label{P2} For any closed pair $(X,Z)$ of algebraic schemes i.e. $Z\overset{\iota}{\hookrightarrow} X$ is a closed subscheme of $X$ of codimension $c:=\dim X-\dim Z$ and with complement $U\overset{j}{\hookrightarrow} X$ satisfying $\dim U=\dim X$, we have a localisation sequence 
    \[\cdots \longrightarrow H^{i-2c}(Z,R(j-c))\overset{\iota_{\ast}}{\longrightarrow} H^{i}(X,R(j))\overset{j^{\ast}}{\longrightarrow}H^{i}(U,R(j))\overset{\partial}{\longrightarrow}H^{i+1-2c}(Z,R(j-c))\longrightarrow \cdots,\] where $\partial$ is called the residue map. The localisation sequence is contravariantly functorial with respect to open immersions $(V,V\cap Z)\hookrightarrow (X,Z)$ and covariantly functorial with respect to cartesian proper morphisms of closed pairs $f:(X,Z)\to(Y,T)$, i.e. $Z=f^{-1}(T)$.
    \item\label{P3} For any algebraic scheme $X$, there is a natural long exact Bockstein sequence \[\cdots \longrightarrow H^{i}(X,\Z_{\ell}(i))\overset{\times \ell^{r}}{\longrightarrow} H^{i}(X,\Z_{\ell}(i))\longrightarrow H^{i}(X,\Z/\ell^{r}(i))\overset{\delta}{\longrightarrow} H^{i+1}(X,\Z_{\ell}(i))\longrightarrow\cdots,\] where $H^{i}(X,\Z_{\ell}(i))\to H^{i}(X,\Z/\ell^{r}(i))$ is the map obtained by functoriality in the coefficients and where the boundary map $\delta$ is called the Bockstein map. The sequence is contravariantly functorial with respect to \'etale maps $j\colon U\to X$ with $\dim U=\dim X$ and covariantly functorial with respect to proper morphisms $f\colon X\to Y$.
 \end{enumerate}
  \end{proposition}
  \begin{proof} The properties \ref{P1} and \ref{P2} listed in the proposition are immediate consequences of the fact that the triangulated premotivic category of $\h$-motives $\DM_{\h}(-,R)$ satisfies Grothendieck's $6$ functors formalism (see \cite[Thm. 5.6.2]{cd} and \cite[Def. A.1.10]{cd}). For \ref{P3} one notes that $\mathbbm{1}_{\Spec(k)}=\Z_{\ell}(0)$ is simply the constant $h$-sheaf associated to the ring $\Z_{\ell}$. Thus we obtain a distinguished triangle \begin{align}\label{eq:Bockstein-sequence}\pi^{!}_{X}\Z_{\ell}(0)\overset{\times \ell^{r}}{\longrightarrow}\pi^{!}_{X}\Z_{\ell}(0)\longrightarrow\pi^{!}_{X}\Z/\ell^{r}\longrightarrow \pi_{X}^{!}\Z_{\ell}(0)[1]\end{align} in $\DM_{\h}(X,\Z_{\ell})$ that is induced from the relevant short exact sequence in $\Sh_{\h} (k,\Z_{\ell})$. Finally, we apply the cohomological functor \[H^{i}_{\BMM}(X_{\et},-(j)):=\Hom_{\DM_{\h}(X,\Z_{\ell})}(\Z_{\ell}(d_{X}-j)[2d_{X}-i],(-))\] to \eqref{eq:Bockstein-sequence} to get the desired long exact sequence.\end{proof}
  Let $X$ be an algebraic scheme and let $x\in X$ be any point. In what follows, we adopt the notation: \begin{align}\label{def:motivic-cohomology-of-point} H^{i}(x,R(j)):=\underset{\emptyset\neq U\subset\overline{\{x\}}}{\colim}H^{i}_{\BMM}(U_{\et},R(j)),\end{align} where $U$ runs through the non-empty open subsets of $\overline{\{x\}}$. We shall use the following simple calculations: \begin{lemma}\label{lemma:cohomology-of-point} Let $X$ be an algebraic $k$-scheme and let $\ell$ be a prime invertible in $k$. Fix a point $x\in X$. For $R\in \{\Z/\ell^{r},\Z_{\ell},\Q_{\ell}\}$, one has: \begin{enumerate}[label=$(\roman*)$]
      \item\label{P4} $H^{i}(x,R(j))\cong H^{i}_{L}(k(x),R(j))$, where $k(x)$ is the residue field of the point $x\in X$.
      \item\label{P5} $H^{1}(x,\Z_{\ell}(1))\cong k(x)^{\ast}\otimes\Z_{\ell}$ and $H^{1}(x,\Z/\ell^{r}(1))\cong k(x)^{\ast}/\ell^{r}$.
      \item \label{P6} $H^{i}(x,R(j))=0$, when $i<0$ and $j\leq 0$. Furthermore, $H^{i}(x,\Q_{\ell}(j))=0$ for all $i>j$.
      \item\label{P7} There exists a distinguished class $[x]\in H^{0}(x,R(0))$, such that $H^{0}(x,R(0))=R[x]$ and such that the isomorphism $R\cong H^{0}(x,R(0)),1 \to [x]$ is functorial in the coefficient $R$.
  \end{enumerate}\end{lemma}\begin{proof} Let $k':=k^{\text{perf}}$ denote the perfect closure of the field $k$. Since the natural map $p:X_{k'}\to X$ is a universal homeomorphism, Lemma \ref{lemma:topological-invariance} yields that the colimit in \eqref{def:motivic-cohomology-of-point} can be taken over the smooth non-empty open subsets of $\overline{\{x'\}}$, where $x'$ is the unique point of $X_{k'}$ lying over $x\in X$. By \eqref{etale-vs-lichtenbaum}, we deduce \[H^{i}(x,R(j)):=\underset{\emptyset\neq U\subset\overline{\{x'\}}}{\colim}H^{i}_{L}(U,R(j))\cong H^{i}_{L}(k'(x'),R(j)),\] where as before the colimit runs through the non-empty smooth open subsets of $\overline{\{x'\}}$ and where the last isomorphism follows from \cite[Lemma 3.9]{mvw}. The field extension $k(x)\subset k'(x')$ is purely inseperable and thus \eqref{etale-vs-lichtenbaum} together with Lemma \ref{lemma:topological-invariance} also imply that $H^{i}_{L}(k(x),R(j))\cong H^{i}_{L}(k'(x'),R(j))$. This proves \ref{P4}. \par To see \ref{P5}, we use \ref{P4} as follows. Lichtenbaum cohomology is given as the \'etale hypercohomology of Bloch's cycle complex $R_{X}(j)=\mathcal{Z}^{j}(-,\ast)_{\et}[-2j]\otimes^{\mathbb{L}}R$. The complex $R_{X}(1)$ is known to be quasi-isomorpic to $\mathbb{G}_{m}[-1]\otimes^{\mathbb{L}}R$ (cf. \cite[Thm. 4.1]{mvw}). In addition, the Kummer sequence provides with the quasi-isomorphism $\mathbb{G}_{m}[-1]\otimes^{\mathbb{L}}\Z/\ell^{r}\cong\mu_{\ell^{r}}$. The result follows.\par To prove \ref{P6}, we look at the exact sequence induced by the Bockstein sequence (see Prop. \ref{prop:twisted-borel-moore-axioms} \ref{P3}): \[H^{i-1}(x,\Q_{\ell}/\Z_{\ell}(j))\longrightarrow H^{i}(x,\Z_{\ell}(j))\longrightarrow H^{i}(x,\Q_{\ell}(j)).\] Thus it suffices to show that $H^{i}(x,R(j))=0$ for $i<0$ and $j\leq 0$, when $R\in\{\Z/\ell^{r},\Q_{\ell}\}$. We clearly have $H^{i}_{L}(k(x),\Z/\ell^{r}(j))=H^{i}(k(x)_{\et},\mu_{\ell^{r}}^{\otimes j})=0$ for all $i<0$. Furthermore, by \eqref{etale-vs-motivic}, the natural map $$H^{i}_{M}(k(x),\Q_{\ell}(j))\longrightarrow H^{i}_{L}(k(x),\Q_{\ell}(j))$$ is an isomorphism. The first vanishing result thus follows from \cite[Cor. 2]{voe}. Concerning the second claim of \ref{P6}, we use \cite[Cor. 2]{voe} once again to conclude $$H^{i}_{M}(k(x),\Q_{\ell}(j))=\CH^{j}(\Spec k(x), 2j-i; \Q_{\ell})=0$$ for all $i>j$. \par Finally, we show \ref{P7}. By Lemma \ref{lemma:topological-invariance}, we may assume that the ground field $k$ is perfect. For any smooth $k$-scheme $X$, we have a natural isomorphism $\CH^{0}(X)\otimes R\cong H^{0}_{M}(X,R(0))=H^{0}_{L}(X,R(0))$. Thus for any non-empty and smooth open subset $U\subset\overline{\{x\}}$, there is a canonical fundamental class $[U]\in H^{0}_{L}(U,R(0))$ with $H^{0}_{L}(U,R(0))=R[U]$ and where the class $[U]$ is compatible with restriction to open subsets. Hence, it provides with a canonical class $[x]\in H^{0}(x,R(0))$ in the limit, as required.\end{proof}

Let $X$ be a variety over $k$ with generic point $\eta\in X$ and let $x\in X^{(1)}$. By virtue of the localisation sequence (see Prop. \ref{prop:twisted-borel-moore-axioms} \ref{P2}), one obtains a map 
\begin{align}\label{def:residue-map} 
\partial: k(\eta)^{\ast}\otimes\Z_{\ell}=H^{1}(\eta,\Z_{\ell}(1))\longrightarrow H^{0}(x,\Z_{\ell}(0))=\Z_{\ell}[x].\end{align} Simply by functoriality it coincides with the composite (cf. Lemma \ref{lemma:comparison maps} below) \[k(\eta)^{\ast}\otimes\Z_{\ell}\overset{\epsilon}{\longrightarrow}H^{1}_{cont}(\eta,\Z_{\ell}(1))\overset{\partial}{\longrightarrow}\Z_{\ell}[x],\] where $\epsilon$ is the map induced from the Kummer sequence (see \cite[(6.3)]{Sch-refined}). It follows that for a regular point $x\in X^{(1)}$, the map \eqref{def:residue-map} is just $-\nu_{x}$, where $\nu_{x}$ stands for the valuation on $k(\eta)$ given by $x$ (see \cite[Prop. 6.6]{Sch-refined}).  
\begin{remark}\label{rem:Borel-Moore-axioms}Following \cite[Def. 4.2 and Def. 4.4]{Sch-refined}, the above (Prop. \ref{prop:twisted-borel-moore-axioms} and Lemma \ref{lemma:cohomology-of-point}) show that the Borel-Moore motivic \'etale cohomology $H^{i}_{\BMM}({-}_{\et},R(j))$, $R\in\{\Z/\ell^{r},\Z_{\ell},\Q_{\ell}\}$, where $\ell$ is a prime invertible in $k$ and $i,j\in\Z$ forms an $\ell$-adic twisted Borel-Moore cohomology theory on the category of algebraic $k$-schemes. To be precise, the property $(P4)$ from \cite[Def. 4.2]{Sch-refined}, i.e. $H^{i}(x,R(j))=0$ for $i<0$ has been substituted by the weaker condition of Lemma \ref{lemma:cohomology-of-point} \ref{P6} which suffices for our purposes (see Cor. \ref{cor:vanishing-consequence} below). Furthermore, we note that the Beilinson-Soul\'e vanishing conjecture \cite{Be} predicts that $(P4)$ holds. We could possibly use the set of coefficients $\{\Z/n,\Z[\frac{1}{p}],\Q\}$, instead, where $p$ is the characteristic exponent of $k$. We prefer to work $\ell$-adically, as we compare this theory with $\ell$-adic Borel-Moore pro-\'etale cohomology \cite[Sec. 6.1]{Sch-refined} (see subsection \ref{sub:3} below).\end{remark}
For $s\geq0$, we let $F_{s}X:=\{x\in X |\dim X-\dim\overline{\{x\}}\leq s\}$ and define as in \cite[Sec. 5]{Sch-refined} \begin{align}
    H^{i}_{\BMM}({F_{s}X}_{\et},R(j)):=\underset{F_{s}X\subset U\subset X}{\colim}H^{i}_{\BMM}(U_{\et},R(j)),
\end{align} where the colimit is taken over all open subsets $U\subset X$ with $F_{s}X\subset U$. The consequences of the localisation sequence (Prop. \ref{P2}) from \cite[Sec. 5.2.]{Sch-refined} hold true also in our situation. In particular, we have a long exact sequence (see \cite[Lemma 5.8]{Sch-refined}) \begin{align}\label{def:Gysin-sequence} \longrightarrow
   H^i_{\BMM}({F_sX}_{\et},R(j))\longrightarrow   H^i_{\BMM}({F_{s-1}X}_{\et},R(j))  \stackrel{\del}\longrightarrow \bigoplus_{x\in X^{(s)}}H^{i+1-2s}(x,R(j-s)) \stackrel{\iota_\ast} \longrightarrow    H^{i+1}_{\BMM}({F_sX}_{\et},R(j)) 
\end{align} Note that the lower bound given in the following Corollary is different from the one used in \cite[Cor. 5.10]{Sch-refined}.\begin{corollary}\label{cor:vanishing-consequence} Let $X$ be an algebraic scheme over a field $k$ and let $\ell$ be a prime invertible in $k$. Let $R\in\{\Z/\ell^{r},\Z_{\ell},\Q_{\ell}\}$. Then $H^{i}_{\BMM}({F_{s}X}_{\et},R(j))\cong H^{i}_{\BMM}(X_{\et},R(j))$ for all $s\geq\max\{\lceil i/2\rceil,j-1\}$.\end{corollary}
\begin{proof} The localisation sequence \eqref{def:Gysin-sequence} together with the first vanishing result of Lemma \ref{lemma:cohomology-of-point} \ref{P6} imply that $H^{i}_{\BMM}({F_{s}X}_{\et},R(j))\cong H^{i}_{\BMM}({F_{s+1}X}_{\et},R(j))$ for all $s\geq\max\{\lceil i/2\rceil,j-1\}$. Since $F_{s}X=X$ for $s\geq\dim X$, the result follows.\end{proof} 
Let $\ell$ be a prime invertible in $k$. By \cite[Lemma 7.1]{Sch-refined} and Lemma \ref{lemma:comparison maps} below, we deduce the following description of the $\ell$-adic Chow group  \begin{align}\label{def:ell-adic-Chow-group}\CH^{i}(X)_{\Z_{\ell}}=\frac{\bigoplus_{x\in X^{(i)}}\Z_{\ell}[x]}{\bigoplus_{x\in X^{(i-1)}}H^{1}(x,\Z_{\ell}(1))\overset{\partial\circ\iota_{\ast}}{\longrightarrow}\bigoplus_{x\in X^{(i)}}\Z_{\ell}[x]},\end{align} where $\partial\circ\iota_{\ast}$ is simply the composition \begin{align}\label{eq:comp}\bigoplus_{x\in X^{(i-1)}}H^{1}(x,\Z_{\ell}(1))\overset{\iota_{\ast}}{\longrightarrow}H_{\BMM}^{2i-1}({F_{i-1}X}_{\et},\Z_{\ell}(i))\overset{\partial}{\longrightarrow}\bigoplus_{x\in X^{(i)}}\Z_{\ell}[x].\end{align} Consider the following pushforward map that comes from the localisation sequence \eqref{def:Gysin-sequence}\begin{align}\label{eq:pushforward}\bigoplus_{x\in X^{(i)}}H^{0}(x,\Z_\ell(0))= \bigoplus_{x\in X^{(i)}} \Z_\ell[x] \stackrel{\iota_\ast} \longrightarrow    H^{2i}_{\BMM}({F_iX}_{\et},\Z_\ell(i))=H^{2i}_{\BMM}(X_{\et},\Z_\ell(i))\end{align} and where for the last equality we used Corollary \ref{cor:vanishing-consequence}. We obtain the following cycle class map.
\begin{lemma} Let $X$ be an algebraic $k$-scheme and let $\ell$ be a prime invertible in $k$. There is a well-defined cycle class map \begin{align}\label{def:kahn-map}
    \tilde{\alpha}^{i}_{X}\colon\CH^{i}(X)_{\Z_{\ell}}\longrightarrow H^{2i}_{\BMM}(X_{\et},\Z_\ell(i)).
\end{align}
\end{lemma}
\begin{proof} We need to show that the map \eqref{eq:pushforward} factors through the $\ell$-adic Chow group $\CH^{i}(X)_{\Z_{\ell}}$. By \eqref{def:ell-adic-Chow-group}, it suffices to show that the composition of $\iota_{\ast}$ from \eqref{eq:pushforward} with the map of \eqref{eq:comp} gives the zero map. The latter is evident by exactness of the localisation sequence \eqref{def:Gysin-sequence}.\end{proof}
\begin{corollary}\label{cor:etale-motivic-cycle-class-map} Let $X$ be smooth and equi-dimensional algebraic $k$-scheme. Then the \'etale motivic cycle class map \eqref{def:motivic-cycle-map} agrees with \eqref{def:kahn-map}.\end{corollary}
\begin{proof} Let $X$ be a smooth and equi-dimensional algebraic scheme over $k$. Lemma \ref{lemma:topological-invariance}, allows one to replace $k$ with its perfect closure. By \eqref{etale-vs-lichtenbaum}, we have \[H^{2i}_{\BMM}(X_{\et},\Z_\ell(i))\cong H^{2i}_{L}(X,\Z_\ell(i)).\] Pick a cycle $z\in\CH^{i}(X)_{\Z_{\ell}}$ and let $Z:=\supp(z)\subset X$. Let $S\subset Z$ denote the singular locus of $Z$. Since $\codim_{X}(S)\geq i+1$, we find from the relative localization sequences the following isomorphisms \[\CH^{i}(X)\cong\CH^{i}(X\setminus S)\ \text{and}\ H^{2i}_{L}(X,\Z_\ell(i))\cong H^{2i}_{L}(X\setminus S,\Z_\ell(i)).\] We consider the following diagram: \begin{center}
    \begin{tikzcd}
{\CH^{0}(Z^{\sm})_{\Z_{\ell}}=\bigoplus_{x\in Z^{(0)}}\Z_{\ell}[x]} \arrow[d, "\alpha^{i}_{Z^{\sm}}", no head, equal] \arrow[r, "\iota_{\ast}"] & \CH^{i}(X\setminus S) \arrow[d, "\alpha^{i}_{X\setminus S}"] \arrow[r, "\cong"] & \CH^{i}(X) \arrow[d, "\alpha^{i}_{X}"] \\
{H^{0}_{L}(Z^{\sm},\Z_{\ell}(0))} \arrow[r, "\iota_{\ast}"]                                                                                       & {H^{2i}_{L}(X\setminus S,\Z_{\ell}(i))} \arrow[r, "\cong"]                      & {H^{2i}_{L}(X,\Z_{\ell}(i)).}          
\end{tikzcd}
\end{center}The commutativity of the diagram above shows that $\alpha^{i}_{X}$ is also induced from the pushforward \eqref{eq:pushforward}. In particular, we have the equality $\alpha^{i}_{X}=\tilde{\alpha}^{i}_{X}$, as claimed.\end{proof}

\begin{corollary}\label{cor:rational-equivalence} Let $X$ be an algebraic $k$-scheme and let $\ell$ be a prime invertible in $k$. Then the cycle class map $\tilde{\alpha}^{1}_{X}$ from \eqref{def:kahn-map} is an isomorphism. In particular, the $\ell$-adic twisted Borel-Moore cohomology theory $H^{i}_{\BMM}({-}_{\et},R(j))$ is adapted to rational equivalence (see \cite[Def. 4.5]{Sch-refined}).\end{corollary}
\begin{proof} From the localisation sequence \eqref{def:Gysin-sequence}, we find the following exact sequence:\[H^{1}_{\BMM}({F_{0}X}_{\et},\Z_{\ell}(1))\overset{\partial}{\longrightarrow}\bigoplus_{x\in X^{(1)}}\Z_{\ell}[x]\overset{\iota_{\ast}}{\longrightarrow}H^{2}_{\BMM}(X_{\et},\Z_{\ell}(1))\longrightarrow H^{2}_{\BMM}({F_{0}X}_{\et},\Z_{\ell}(1)).\] The second isomorphism in Lemma \ref{lemma:topological-invariance} together with \cite[Lemma 5.6]{Sch-refined} yield $$H^{i}_{\BMM}({F_{0}X}_{\et},\Z_{\ell}(j))=\bigoplus_{\eta\in X^{(0)}} H^{i}(\eta,\Z_{\ell}(j)).$$ Hence, by \eqref{def:ell-adic-Chow-group}, the cokernel of $\partial$ in the above sequence is simply the $\ell$-adic Chow group $\CH^{1}(X)_{\Z_{\ell}}$. This shows that $\tilde{\alpha}^{1}_{X}$ is injective. On the other hand, $H^{2}(\eta,\Z_{\ell}(1))=H^{2}_{L}(k(\eta),\Z_{\ell}(1))\cong \CH^{1}(\Spec k(\eta))=0$ (simply because $\Z_{\ell}(1)$ is quasi-isomorphic to $\mathbb{G}_{m}[-1]\otimes\Z_{\ell}$), implies that $\tilde{\alpha}^{1}_{X}$ is also surjective.\end{proof}
 \subsection{Comparison to pro-\'etale Borel-Moore Cohomology}\label{sub:3} Let $X$ be an algebraic $k$-scheme and let $\ell$ be a prime invertible in $k$. For a ring $R\in\{\Z/\ell^{r},\Z_{\ell},\Q_{\ell}\}$, we consider the twisted Borel-Moore pro-\'etale cohomology $H^{i}_{\BM}(X_{\proet},R(j))$ of $X$, as stated in (\cite[(6.13)-(6.15)]{Sch-refined}, \cite[Prop. 6.6]{Sch-refined}), where $X_{\proet}$ stands for the pro-\'etale site of $X$ (see \cite{BS}). Some of the main properties of this functor are contained in \cite[Sec. 4]{Sch-refined}. We recall that if $X$ is smooth and equi-dimensional, then we have the following natural isomorphisms: \begin{align}\label{eq:BM-proet-smooth-case-1} H^{i}_{\BM}(X_{\proet},\Z/\ell^{r}(j))\cong H^{i}(X_{\et},\mu_{\ell^{r}}^{\otimes j})\ \text{and}\ H^{i}_{\BM}(X_{\proet},\Z_{\ell}(j))\cong H^{i}_{cont}(X_{\et},\Z_{\ell}(j)),
\end{align}
where $H^{i}_{cont}$ denotes Jannsen's continuous \'etale cohomology (see \cite{jannsen}).
\begin{lemma}\label{lemma:comparison maps} Let $k$ be a field and let $\ell$ be a prime invertible in $k$. For any algebraic $k$-scheme $X$ and any ring $R\in\{\Z/\ell^{r},\Z_{\ell},\Q_{\ell}\}$, there are canonical homomorphisms 
\begin{align}
    \label{def:comparison-map} \rho^{i,j}_{X}\colon H^{i}_{\BMM}(X_{\et},R(j))\longrightarrow H^{i}_{\BM}(X_{\proet},R(j))
\end{align}
that are compatible with \'etale maps $j\colon U\to X$ $(\text{where}\ \dim U=\dim X)$
and proper morphisms $f\colon X\to Y$. Moreover, if $R=\Z/\ell^{r}$, then $\rho^{i,j}_{X}$ is an isomorphism.\end{lemma}
 \begin{proof} If $R=\Z/\ell^{r}$, then $H^{i}_{\BM}(X_{\proet},\Z/\ell^{r}(j))$ is Borel-Moore \'etale cohomology (see \cite{BS}) and the result thus follows from \eqref{etale=motivic}. Clearly it suffices to restrict to $R=\Z_{\ell}$, as the remaining case i.e. $R=\Q_{\ell}$, then follows from tensoring $\rho^{i,j}_{X}$ with $\Q_{\ell}$. By \cite[Thm. 7.2.11]{cd} and \cite[Prop. 7.2.21]{cd}, there is a premotivic adjunction \begin{center}
     \begin{tikzcd}
{\hat{r}^{\ast}_{\ell}:\DM_{\h}(X,\Z_{\ell})} \arrow[r, shift right] & {\D_{Ek}(X_{\et},\Z_{\ell}):\hat{r}_{\ell\ast}}, \arrow[l, shift right]
\end{tikzcd} 
 \end{center}
 where $\D_{Ek}(X_{\et},\Z_{\ell})$ is Ekedahl's triangulated monoidal category of $\ell$-adic systems (\cite[Def. 2.5]{Eke}) and where the premotivic morphism $\hat{r}^{\ast}_{\ell}$ is compatible with the six operations. Hence, we obtain a natural homomorphism $H^{i}_{\BMM}(X_{\et},\Z_{\ell}(j))\to H^{i}_{\BM,Ek}(X_{\et},\Z_{\ell}(j))$, where the target group is defined analogously to \eqref{def:Borel-Moore-Motivic}. To conclude, we note that there is a canonical isomorphism \[H^{i}_{\BM,Ek}(X_{\et},\Z_{\ell}(j))\cong H^{i}_{\BM}(X_{\proet},\Z_{\ell}(j))\] by the comparison result \cite[Prop. 5.5.4]{BS}.\end{proof}
 We list some of the properties of the natural maps from \eqref{def:comparison-map}. \begin{lemma}\label{lemma:prop-rho} Let $X$ be an algebraic $k$-scheme and let $\ell$ be a prime invertible in $k$. The canonical homomorphisms $\rho^{i,j}_{X}$ from \eqref{def:comparison-map} satisfy the following:\begin{enumerate}[label=$(\roman*)$]
     \item \label{rho-1} For all $i,j\in \Z$, the natural map $$\rho_{X,\tors}^{i,j}:H^{i}_{\BMM}(X_{\et},\Z_{\ell}(j))_{\tors}\longrightarrow H^{i}_{\BM}(X_{\proet},\Z_{\ell}(j))_{\tors}$$ is surjective. 
     \item\label{rho-2}Furthermore, if the field $k$ is finitely generated, then $\rho^{2,1}_{X,\tors}$ is an isomorphism.
     \item\label{rho-3} For all $i,j\in \Z$, $\rho^{i,j}_{X}$ has torsion free cokernel and divisible kernel.
 \end{enumerate} 
\end{lemma}
\begin{proof} For \ref{rho-1}, we consider the following commutative square
\begin{center}
    \begin{tikzcd}
{H^{i-1}_{\BMM}(X_{\et},\Q_{\ell}/\Z_{\ell}(j))} \arrow[r, "\delta", two heads] \arrow[d, "{\rho^{i-1,j}_{X}}", equals] & {H^{i}_{\BMM}(X_{\et},\Z_{\ell}(j))_{\tors}} \arrow[d, "{\rho^{i,j}_{X,\tors}}"] \\
{H^{i-1}_{\BM}(X_{\proet},\Q_{\ell}/\Z_{\ell}(j))} \arrow[r, "\delta", two heads]                                    & {H^{i}_{\BM}(X_{\proet},\Z_{\ell}(j))_{\tors}.}    \end{tikzcd}\end{center} The horizontal arrows are simply the Bockstein maps (see Prop. \ref{prop:twisted-borel-moore-axioms} \ref{P3}) and thus give surjections onto the torsion subgroups of their codomains. The map $\rho^{i-1,j}_{X}$ in the above diagram is the identity by Lemma \ref{lemma:comparison maps}. This proves \ref{rho-1}.\par Concerning \ref{rho-2}, we note that the compatibility of $\rho^{i,j}_{X}$ with proper pushforward yields $\rho^{2,1}_{X}\circ\tilde{\alpha}^{1}_{X}=\cl^{1}_{X}$, where $\tilde{\alpha}^{1}_{X}$ and $\cl^{1}_{X}$ are the cycle class maps from \eqref{def:kahn-map} and \cite[(7.1)]{Sch-refined}, respectively. Corollary \ref{cor:rational-equivalence} gives that $\tilde{\alpha}^{1}_{X}$ is an isomorphism. On the other hand, if the field $k$ is finitely generated, then $\cl^{1}_{X}$ is injective and it induces an isomorphism on torsion subgroups. This follows from the same argument used in the proof of Corollary \ref{cor:rational-equivalence}. Namely, we look at the exact sequence \[H^{1}_{\BM}({F_{0}X}_{\proet},\Z_{\ell}(1))\overset{\partial}{\longrightarrow}\bigoplus_{x\in X^{(1)}}\Z_{\ell}[x]\overset{\iota_{\ast}}{\longrightarrow}H^{2}_{\BM}(X_{\proet},\Z_{\ell}(1))\longrightarrow H^{2}_{\BM}({F_{0}X}_{\proet},\Z_{\ell}(1)).\] By \cite[Def. 7.2 and Lemma 7.4]{Sch-refined}, the cokernel of the map $\partial$ is $A^{1}(X)_{\Z_{\ell}}=(\CH^{1}(X)/\sim_{\alg})\otimes \Z_{\ell}$, i.e. the $\ell$-adic Chow group modulo algebraic equivalence. If the field $k$ is finitely generated, we find from \cite[Lemma 7.5 and Prop. 6.6]{Sch-refined} that $$A^{1}(X)_{\Z_{\ell}}=\CH^{1}(X)_{\Z_{\ell}}.$$ In particular, $\cl^{1}_{X}$ is injective. Since $H^{2}_{\BM}({F_{0}X}_{\proet},\Z_{\ell}(1))_{\tors}=0$ (see \cite[Lemma 5.13]{Sch-refined}), we deduce that $\cl^{1}_{X,\tors}$ is an isomorphism, as claimed. The proof of \ref{rho-2} is complete.\par The result of \ref{rho-3} is well-known if $X$ is smooth and equi-dimensional due to Kahn \cite[Cor. 3.5.]{kahn} (cf. \cite{RS16}). Same reasoning gives the general case: We consider once again Ekedahl's triangulated monoidal category of $\ell$-adic systems $\D_{Ek}(X_{\et},\Z_{\ell})$ (see \cite[Def. 2.5]{Eke}). Recall from the proof of Lemma \ref{lemma:comparison maps}, the premotivic adjunction \begin{center}
     \begin{tikzcd}
{\hat{r}^{\ast}_{\ell}:\DM_{\h}(-,\Z_{\ell})} \arrow[r, shift right] & {\D_{Ek}(-_{\et},\Z_{\ell}):\hat{r}_{\ell\ast}}, \arrow[l, shift right]
\end{tikzcd} 
 \end{center} where $\hat{r}^{\ast}_{\ell}$ is compatible with the six operations and consider the natural map $\epsilon_{\ell}:\Z_{\ell}(0)\longrightarrow\hat{r}_{\ell\ast}\hat{r}^{\ast}_{\ell}\Z_{\ell}(0)$ in $\DM_{\h}(k,\Z_{\ell})$, i.e. the unit of the relevant adjunction. Note that the cone $K$ of $\epsilon_{\ell}$ in $\DM_{\h}(k,\Z_{\ell})$ is uniquely $\ell$-divisible (cf. \cite[pp. 92 (7.2.4.c)]{cd}), i.e. lies in the essential image of the inclusion functor $$\DM_{\h}(k,\Q_{\ell})\longrightarrow\DM_{\h}(k,\Z_{\ell}).$$ We set $\hat{\Z}_{\ell}(0):=\hat{r}^{\ast}_{\ell}\Z_{\ell}(0)$. Since $\hat{r}_{\ell\ast}$ commutes with the right adjoint functors of the $6$ functors formalism, we obtain a distinguished triangle \begin{align}\label{eq:dist-trinagle}
     \pi^{!}_{X}\Z_{\ell}(0)\longrightarrow \hat{r}_{\ell\ast}\pi^{!}_{X}\hat{\Z}_{\ell}(0)\longrightarrow \pi^{!}_{X}K\longrightarrow  \pi^{!}_{X}\Z_{\ell}(0)[1].
 \end{align}
 Applying $\Hom_{\DM_{h}(X,\Z_{\ell})}(\Z_{\ell}(d_{X}-j)[2d_{X}-i],-)$ to \eqref{eq:dist-trinagle}, we find from the long exact sequence by using that $K$ is uniquely $\ell$-divisible that $\rho^{i,j}_{X}$ has $\ell$-divisible kernel and torsion free cokernel. This finishes the proof.\end{proof}
 \section{Reformulating the \'etale motivic cycle class map} In this section, we give another construction for the \'etale motivic cycle class map on torsion cycles \begin{align}\label{eq:alpha_{X}} \alpha^{i}_{X,\tors}:\CH^{i}(X)[\ell^{\infty}]\longrightarrow H^{2i}_{L}(X,\Z_{\ell}(i))_{\tors}.\end{align} This allows us to relate \eqref{eq:alpha_{X}} with the map $\lambda^{i}_{X}$ from \cite{alexandrou-schreieder}. The idea is to refine $\alpha^{i}_{X,\tors}$ via a cycle class map $\beta^{i}_{X}$ in exactly the same way that the authors in \cite{alexandrou-schreieder} factored Jannsen's map $\cl^{i}_{X}$ on torsion cycles through $\lambda^{i}_{X}$. \par Recall from \cite[Sec. 5 (5.1)]{Sch-refined} the coniveau filtration $N^{\ast}$ on $H^{i}_{\BMM}(X_{\et},R(j))$ defined by \begin{align}\label{eq:coniveau-filtration}N^{s}H^{i}_{\BMM}(X_{\et},R(j)):=\ker(H^{i}_{\BMM}(X_{\et},R(j))\longrightarrow H^{i}_{\BMM}({F_{s-1}X}_{\et},R(j))).\end{align} 
    The following Lemma is taken from \cite{alexandrou-schreieder}.
    \begin{lemma}\label{lemma:coniveau} Let $X$ be an algebraic $k$-scheme and let $\ell$ be a prime invertible in $k$. Fix a ring $R\in\{\Z/\ell^{r},\Z_{\ell},\Q_{\ell}\}$. The natural restriction map $H^{2i-1}_{\BMM}(X_{\et},R(i))\to H^{2i-1}_{\BMM}({F_{i-1}X}_{\et},R(i))$ is injective. Moreover, $N^{i-1}H^{2i-1}_{\BMM}(X_{\et},R(i))$ as a subgroup of $H^{2i-1}_{\BMM}({F_{i-1}X}_{\et},R(i))$ coincides with \begin{align}\label{eq:skata}\iota_\ast \left( \ker\left(\del\circ \iota_\ast:\bigoplus_{x\in X^{(i-1)}}H^1(x,R(1)) \longrightarrow  \bigoplus_{x\in X^{(i)}}R[x]\right) \right)\subset H^{2i-1}_{\BMM}({F_{i-1}X}_{\et},R(i)).\end{align}\end{lemma}\begin{proof} This is the \'etale Borel-Moore motivic version of \cite[Lemma 3.1]{alexandrou-schreieder}. The proof remains the same.\end{proof}
    We shall need the following simple observation.
    \begin{lemma} Let $X$ be an algebraic $k$-scheme and let $\ell$ be a prime invertible in $k$. Fix $i,j\in \Z$ with $i>j$. We have \begin{align} \label{eq:vanishing} H^{i}_{\BMM}({F_{s}X}_{\et},\Q_{\ell}(j))=0\ \text{for all}\ 0\leq s<i-j.\end{align} In particular, \begin{align}\label{lemma:N^{i-1(Q)}}N^{s+1}H^{i}_{\BMM}(X_{\et},\Q_{\ell}(j))=H^{i}_{\BMM}(X_{\et},\Q_{\ell}(j))\ \text{for all}\ 0\leq s<i-j.\end{align}\end{lemma}
    \begin{proof} Clearly, it suffices to prove only \eqref{eq:vanishing}. The localisation sequence \eqref{def:Gysin-sequence} together with the second vanishing result in Lemma \ref{lemma:cohomology-of-point} \ref{P6}, yield isomorphisms \[H^{i}_{\BMM}({F_{s}X}_{\et},\Q_{\ell}(j))\cong H^{i}_{\BMM}({F_{s-1}X}_{\et},\Q_{\ell}(j))\] for all $1\leq s<i-j$. By \cite[Lemma 5.6]{Sch-refined}, the second isomorphism in Lemma \ref{lemma:topological-invariance} gives $$H^{i}_{\BMM}({F_{0}X}_{\et},\Q_{\ell}(j))=\bigoplus_{\eta\in X^{(0)}} H^{i}(\eta,\Q_{\ell}(j)).$$ To conclude, observe that the last group is zero by the second vanishing result in Lemma \ref{lemma:cohomology-of-point} \ref{P6}.\end{proof}
    Let $X$ be an algebraic $k$-scheme and let $\ell$ be a prime invertible in $k$. By \cite[Lemma 8.1]{Sch-refined}, there is a canonical isomorphism  
    \begin{align}\label{def:varphi_{r}}
        \varphi_{r}:\CH^{i}(X)[\ell^{r}]\stackrel{\cong}\longrightarrow  \frac{\ker\left(\del\circ \iota_\ast:\bigoplus_{x\in X^{(i-1)}} H^1(x,\Z/\ell^{r}(1))  \longrightarrow  \bigoplus_{x\in X^{(i)}}\Z/\ell^r[x] \right) }{\ker\left(\del\circ \iota_\ast:\bigoplus_{x\in X^{(i-1)}}H^1( x,\Z_{\ell}(1)) \longrightarrow  \bigoplus_{x\in X^{(i)}}\Z_\ell[x] \right)},
     \end{align}
 which one describes as follows: Pick a cycle $[z]\in \CH^{i}(X)[\ell^{r}]$. Then by \eqref{def:ell-adic-Chow-group}, we have that $\ell^{r}z=\partial\circ\iota_{\ast}(\xi)$ for some \[\xi\in \bigoplus_{x\in X^{(i-1)}} H^1(x,\Z_{\ell}(1)).\] We denote by $\bar{\xi}$ the reduction of $\xi$ modulo $\ell^{r}$ and set $\varphi_{r}([z]):=[\bar{\xi}]$.\par Next we construct the homomorphism $\beta^{i}_{X}$ as promised at the beginning of this section.\begin{lemma}\label{lemma:beta^{i}_{X}} Let $X$ be an algebraic $k$-scheme and let $\ell$ be a prime invertible in $k$. There is a well-defined cycle class map $\beta^{i}_{X}$ on torsion cycles \begin{align}\label{def:beta_X}\beta^{i}_{X}:\CH^{i}(X)[\ell^{\infty}]\longrightarrow H^{2i-1}_{\BM}(X_{\et},\Q_{\ell}/\Z_{\ell}(i))/H^{2i-1}_{\BMM}(X_{\et},\Q_{\ell}(i)),\end{align} whose image is given by \[\im(\beta^{i}_{X})=N^{i-1}H^{2i-1}_{\BM}(X_{\et},\Q_{\ell}/\Z_{\ell}(i))/H^{2i-1}_{\BMM}(X_{\et},\Q_{\ell}(i)).\]\end{lemma}
 \begin{proof} By virtue of \eqref{etale=motivic}, we identify Borel-Moore motivic \'etale with Borel-Moore \'etale cohomology whenever $R=\Z/\ell^{r}$.\par We have a commutative square \begin{center}\begin{tikzcd}
{\ker\left(\bigoplus_{x\in X^{(i-1)}}H^1( x,\Z_{\ell}(1)) \overset{\partial\circ \iota_\ast}{\to}  \bigoplus_{x\in X^{(i)}}\Z_\ell[x] \right)} \arrow[r, "\iota_{\ast}", two heads] \arrow[d]  & {N^{i-1}H^{2i-1}_{\BMM}(X_{\et},\Z_{\ell}(i))} \arrow[d] \\
{\ker\left(\bigoplus_{x\in X^{(i-1)}} H^1(x,\Z/\ell^{r}(1)) \overset{\partial\circ \iota_\ast}{\to}  \bigoplus_{x\in X^{(i)}}\Z/\ell^r[x] \right)} \arrow[r, "\iota_{\ast}", two heads] & {N^{i-1}H^{2i-1}_{\BM}(X_{\et},\Z/\ell^{r}(i))},\end{tikzcd}\end{center}
 where the horizontal maps are well-defined and surjective by Lemma \ref{lemma:coniveau} and where the vertical maps are simply given by functoriality in the coefficients. Passing to the cokernels and pre-composing with the isomorphism $\varphi_{r}$ from \eqref{def:varphi_{r}}, we obtain a cycle class map \begin{align}\label{def:beta'_X}\beta^{i}_{X,r}:\CH^{i}(X)[\ell^{r}]\longrightarrow H^{2i-1}_{\BM}(X_{\et},\Z/\ell^{r}(i))/N^{i-1}H^{2i-1}_{\BMM}(X_{\et},\Z_{\ell}(i)),\end{align} whose image is given by \[\im(\beta^{i}_{X,r})=N^{i-1}H^{2i-1}_{\BM}(X_{\et},\Z/\ell^{r}(i))/N^{i-1}H^{2i-1}_{\BMM}(X_{\et},\Z_{\ell}(i)).\] Taking the direct limit of \eqref{def:beta'_X} over all positive integers $r\geq 1$, we obtain the desire homomorphism \eqref{def:beta_X}, once we use $N^{i-1}H^{2i-1}_{\BMM}(X_{\et},\Q_{\ell}(i))=H^{2i-1}_{\BMM}(X_{\et},\Q_{\ell}(i))$ (see \eqref{lemma:N^{i-1(Q)}}).\end{proof}
 The following Lemma shows that $\beta^{i}_{X}$ and $\alpha^{i}_{X,\tors}$ can be identified, when $X$ is a smooth and equi-dimensional algebraic $k$-scheme.
 \begin{lemma}\label{lemma:beta_X=alpha_X} Let $X$ be an algebraic $k$-scheme which admits a closed embedding into a smooth and equi-dimensional algebraic $k$-scheme. Let $\ell$ be a prime invertible in $k$. The following triangle:\begin{center}
     \begin{tikzcd}
                                                                                    & {H^{2i-1}_{\BM}(X_{\et},\Q_{\ell}/\Z_{\ell}(i))/H_{\BMM}^{2i-1}(X_{\et},\Q_{\ell}(i))} \arrow[d, "\delta", "\cong"'] \\
{\CH^{i}(X)[\ell^{\infty}]} \arrow[ru, "\beta^{i}_{X}"] \arrow[r, "\tilde{\alpha}^{i}_{X}"] & {H^{2i}_{\BMM}(X_{\et},\Z_{\ell}(i))[\ell^{\infty}]}                                                
\end{tikzcd}
 \end{center}
 commutes, where $\delta$ is the isomorphism induced from the Bockstein map (see Prop. \ref{prop:twisted-borel-moore-axioms} \ref{P3}) and where $\tilde{\alpha}^{i}_{X}$ and $\beta^{i}_{X}$ are the maps from \eqref{def:kahn-map} and \eqref{def:beta_X}, respectively.
 \end{lemma}
 \begin{proof} To prove commutativity for the triangle in question, i.e. $\delta\circ\beta^{i}_{X}=\tilde{\alpha}^{i}_{X}$, we would like to utilise the analogous result from \cite[Lemma 3.5]{alexandrou-schreieder}. By virtue of Lemma \ref{lemma:colimit}, it suffices to prove the claim in case the field $k$ is finitely generated over its prime field.\par Let $[z]\in\CH^{i}(X)[\ell^{\infty}]$ and choose a closed subset $\iota:W\hookrightarrow X$ of pure codimension $i-1$, such that $[z]\in\im(\CH^{1}(W)[\ell^{\infty}]\overset{\iota_{\ast}}{\to}\CH^{i}(X)[\ell^{\infty}])$. Compatibility with respect to proper pushforward reduces our task to showing that \[\delta\circ\beta^{1}_{W}([z])=\tilde{\alpha}^{1}_{W}([z]).\] We consider the following diagram\begin{center}\label{eq:diagram}
     \begin{tikzcd}
{\CH^{1}(W)[\ell^{\infty}]} \arrow[r, "\beta^{1}_{W}"] \arrow[d, equals]                               & {\frac{H^{1}_{\BM}(W_{\et},\Q_{\ell}/\Z_{\ell}(1))}{H^{1}_{\BMM}(W_{\et},\Q_{\ell}(1))}} \arrow[d, "{-\rho^{1,1}_{W}}"] \arrow[r, "\delta", "\cong"'] & {H^{2}_{\BMM}(W_{\et},\Z_{\ell}(1))_{\tors}} \arrow[d, "{\rho^{2,1}_{W,\tors}}"] \\
{\CH^{1}(W)[\ell^{\infty}]} \arrow[r, "\lambda^{1}_{W}", "\cong"'] \arrow[rr, "{\cl^{1}_{W,\tors}}"', bend right, shift right] & {\frac{H^{1}_{\BM}(W_{\proet},\Q_{\ell}/\Z_{\ell}(1))}{H^{1}_{\BM}(W_{\proet},\Q_{\ell}(1))}} \arrow[r, "-\delta", "\cong"']                 & {H^{2}_{\BM}(W_{\proet},\Z_{\ell}(1))_{\tors},}                       
\end{tikzcd}
\end{center}
 where $\rho^{i,j}_{W}$ are the natural homomorphisms from \eqref{def:comparison-map}. The cycle class map $\lambda^{1}_{W}$ is an isomorphism by \cite[Cor. 4.4]{alexandrou-schreieder}. The commutativity of the first square follows from the construction of $\beta^{1}_{W}$, which is analogous to the one of $\lambda^{1}_{W}$ (see \cite[pp. 6, 13]{alexandrou-schreieder}) as well as from the fact that $\rho^{1,1}_{W}$ is compatible with proper pushforward. In addition, the equality $\cl^{1}_{W,\tors}=-\delta\circ\lambda^{1}_{W}$ holds by \cite[Lemma 3.5]{alexandrou-schreieder}.\par We let $\gamma:=\delta\circ\beta^{1}_{W}$. By Lemma \ref{lemma:prop-rho} \ref{rho-2}, the natural map $\rho^{2,1}_{W,\tors}$ is an isomorphism. Hence $\rho^{2,1}_{W,\tors}\circ\tilde{\alpha}^{1}_{W,\tors}=\cl^{1}_{W,\tors}$ yields  $\tilde{\alpha}^{1}_{W,\tors}=\gamma$. The proof of the Lemma is complete.\end{proof}
 
 \begin{lemma}\label{lemma:codomains} Let $X$ be an algebraic scheme over a finitely generated field $k$ and let $\ell$ be a prime invertible in $k$. For $R\in \{\Q_{\ell}/\Z_{\ell},\Q_{\ell}\}$, the canonical homomorphisms $\rho^{2i-1,i}_{X}:H^{2i-1}_{\BMM}(X_{\et},R(i))\to H^{2i-1}_{\BM}(X_{\proet},R(i))$ from \eqref{def:comparison-map} induce the following isomorphism: \begin{align}\label{eq:codomains-beta_{X}-vs-lambda_{X}} \frac{H^{2i-1}_{\BMM}(X_{\et},\Q_{\ell}/\Z_{\ell}(i))}{H^{2i-1}_{\BMM}(X_{\et},\Q_{\ell}(i))}\cong \frac{H^{2i-1}_{\BM}(X_{\proet},\Q_{\ell}/\Z_{\ell}(i))}{N^{i-1}H^{2i-1}_{\BM}(X_{\proet},\Q_{\ell}(i)).}
 \end{align}\end{lemma}
 \begin{proof} The groups $H^{2i-1}_{\BMM}(X_{\et},\Q_{\ell}/\Z_{\ell}(i))$ and $H^{2i-1}_{\BM}(X_{\proet},\Q_{\ell}/\Z_{\ell}(i))$ can be identified with Borel-Moore \'etale cohomology (see \eqref{etale=motivic} and \cite{BS}) and thus we can drop the subscripts. We set \[I^{i}_{\BMM}(X_{\et}):=\im\left(H^{2i-1}_{\BMM}(X_{\et},\Q_{\ell}(i))\longrightarrow H^{2i-1}(X,\Q_{\ell}/\Z_{\ell}(i))\right)\] and \[I^{i}_{\BM}(X_{\proet}):=\im\left(N^{i-1}H^{2i-1}_{\BM}(X_{\proet},\Q_{\ell}(i))\longrightarrow H^{2i-1}(X,\Q_{\ell}/\Z_{\ell}(i))\right)\] for $i\geq1$. To prove \eqref{eq:codomains-beta_{X}-vs-lambda_{X}}, it is enough to show that $I^{i}_{\BMM}(X_{\et})=I^{i}_{\BM}(X_{\proet})$. The containment “$\subset$” holds trivially as $$N^{i-1}H^{2i-1}_{\BMM}(X_{\et},\Q_{\ell}(i))=H^{2i-1}_{\BMM}(X_{\et},\Q_{\ell}(i))$$ by \eqref{lemma:N^{i-1(Q)}}. The commutative diagrams\begin{center}\begin{tikzcd}
{\underset{Z\subset X,\ \codim_{X}(Z)=i-1}{\colim}H^{1}_{\BMM}(Z_{\et},\Q_{\ell}(1))} \arrow[r, two heads] \arrow[d] & {H^{2i-1}_{\BMM}(X_{\et},\Q_{\ell}(i))} \arrow[d] \\
{\underset{Z\subset X,\ \codim_{X}(Z)=i-1}{\colim}H^{1}(Z,\Q_{\ell}/\Z_{\ell}(1))} \arrow[r, two heads]            & {N^{i-1}H^{2i-1}(X,\Q_{\ell}/\Z_{\ell}(i)),}           
\end{tikzcd}
 \end{center}
and 
\begin{center}\begin{tikzcd}
{\underset{Z\subset X,\ \codim_{X}(Z)=i-1}{\colim}H^{1}_{\BM}(Z_{\proet},\Q_{\ell}(1))} \arrow[r, two heads] \arrow[d] & {N^{i-1}H^{2i-1}_{\BM}(X_{\proet},\Q_{\ell}(i))} \arrow[d] \\
{\underset{Z\subset X,\ \codim_{X}(Z)=i-1}{\colim}H^{1}(Z,\Q_{\ell}/\Z_{\ell}(1))} \arrow[r, two heads]            & {N^{i-1}H^{2i-1}(X,\Q_{\ell}/\Z_{\ell}(i)),}           
\end{tikzcd}
 \end{center}
where the horizontal maps are clearly onto, induce the surjections \[\underset{Z\subset X,\ \codim_{X}(Z)=i-1}{\colim}I^{1}_{\BMM}(Z_{\et})\twoheadrightarrow I^{i}_{\BMM}(X_{\et})\] and \[\underset{Z\subset X,\ \codim_{X}(Z)=i-1}{\colim}I^{1}_{\BM}(Z_{\proet})\twoheadrightarrow I^{i}_{\BM}(X_{\proet}),\] respectively. Hence the task is reduced to showing $I^{1}_{\BMM}(X_{\et})=I^{1}_{\BM}(X_{\proet})$ for an algebraic $k$-scheme $X$ or equivalently the isomorphism \eqref{eq:codomains-beta_{X}-vs-lambda_{X}} for $i=1$. The latter follows from Lemma \ref{lemma:prop-rho} \ref{rho-2} and the fact that the square \begin{center}\begin{tikzcd}
{\frac{H^{1}_{\BM}(X_{\et},\Q_{\ell}/\Z_{\ell}(1))}{H^{1}_{\BMM}(X_{\et},\Q_{\ell}(1))}} \arrow[r, "\cong"] \arrow[d] & {H^{2}_{\BMM}(X_{\et},\Z_{\ell}(1))_{\tors}} \arrow[d, "{\cong}", "{\rho^{2,1}_{X}}"'] \\
{\frac{H^{1}_{\BM}(X_{\et},\Q_{\ell}/\Z_{\ell}(1))}{H^{1}_{\BM}(X_{\proet},\Q_{\ell}(1))}} \arrow[r, "\cong"]            & {H^{2}_{\BM}(X_{\proet},\Z_{\ell}(1))_{\tors}}                    \end{tikzcd}
\end{center} commutes, where the horizontal isomorphisms are induced from the corresponding Bockstein maps. \end{proof} 
The cycle class map $\lambda^{i}_{X}$ from \eqref{def:blochs-map} exists even if we take $X$ to be singular (see \cite[Sec. 4.1]{alexandrou-schreieder}). It takes the form \begin{align}\label{def:Blochs-map-singular} \lambda^{i}_{X}\colon \CH^{i}(X)[\ell^{\infty}]\longrightarrow H^{2i-1}_{\BM}(X_{\et},\Q_{\ell}/\Z_{\ell}(i))/M^{2i-1}(X),
\end{align} 
where the group $M^{2i-1}(X)$ is defined analogously to \eqref{def:M^2i-1} with the only difference that ordinary cohomology is replaced by Borel-Moore pro-\'etale cohomology.
 \begin{corollary}\label{cor:beta_X-vs-lambda_X} Let $X$ be an algebraic $k$-scheme and let $\ell$ be a prime invertible in $k$. Then the cycle class maps $\lambda^{i}_{X}$ from \eqref{def:Blochs-map-singular} and $\beta^{i}_{X}$ from \eqref{def:beta_X} coincide up to a minus sign, i.e. we have a commutative triangle:\begin{center}
     \begin{tikzcd}
{\CH^{i}(X)[\ell^{\infty}]} \arrow[r, "-\beta^{i}_{X}"] \arrow[rd, "\lambda^{i}_{X}"] & {\frac{H^{2i-1}_{\BM}(X_{\et},\Q_{\ell}/\Z_{\ell}(i))}{H^{2i-1}_{\BMM}(X_{\et},\Q_{\ell}(i))}} \arrow[d, equals] \\
                                                                                      & {\frac{H^{2i-1}_{\BM}(X_{\et},\Q_{\ell}/\Z_{\ell}(i))}{M^{2i-1}(X)}.}                                                      
\end{tikzcd}
 \end{center} \end{corollary}
 \begin{proof} The fact that the diagram in question is commutative follows from the construction of $\beta^{i}_{X}$ which is analogous to the one of $\lambda^{i}_{X}$ (see \cite[pp. 6, 13]{alexandrou-schreieder}). Fix a finitely generated subfield $k_{0}\subset k$, so that we can find a model $X_{0}$ of $X$, i.e. $X=X_{0}\times_{k_{0}}k$. The group $M^{2i-1}(X)$ is given explicitly by:\[M^{2i-1}(X)=\im\left( \underset{k_{0}\subset E\subset k}{\colim} N^{i-1} H^{2i-1}_{\BM}({{X_0}_{E}}_{\proet},\Q_\ell(i))\longrightarrow  H^{2i-1}_{\BM}(X_{\et},\Q_\ell/\Z_\ell(i))\right),\]where the direct limit runs over all finitely generated subfields $E\subset k$ with $k_{0}\subset E$. Especially, if the field $k$ is itself finitely generated, we obtain that $M^{2i-1}(X)=N^{i-1} H^{2i-1}_{\BM}(X_{\proet},\Q_\ell(i))$. In this case, the result is an immediate consequence of Lemma \ref{lemma:codomains}. Expressing $\lambda^{i}_{X}$ as a colimit, i.e. $$\lambda^{i}_{X}=\underset{k_{0}\subset E\subset k}{\colim}\lambda^{i}_{{X_{0}}_{E}},$$ we see that the Corollary follows once we have\begin{align} \label{eq:colimit-beta}\beta^{i}_{X}=\underset{k_{0}\subset E\subset k}{\colim}\beta^{i}_{{X_{0}}_{E}},\end{align} where as before the colimit runs through all finitely generated subfields $E\subset k$ with $k_{0}\subset E$. The latter is evident by Lemma \ref{lemma:colimit}.\end{proof}
 \begin{corollary}\label{cor:l=a} Let $X$ be an algebraic $k$-scheme which admits a closed embedding into a smooth and equi-dimensional algebraic $k$-scheme. Let $\ell$ be a prime invertible in $k$. The following diagram \begin{center}
     \begin{tikzcd}
{\CH^{i}(X)[\ell^{\infty}]} \arrow[r, "\tilde{\alpha}^{i}_{X}"] \arrow[d, "-\lambda^{i}_{X}"'] \arrow[rd, "\beta^{i}_{X}"] & {H^{2i}_{\BMM}(X_{\et},\Z_{\ell}(i))_{\tors}}                                                                 \\
{\frac{H^{2i-1}_{\BM}(X_{\et},\Q_{\ell}/\Z_{\ell}(i))}{M^{2i-1}(X)}} \arrow[r, equals]                      & {\frac{H^{2i-1}_{\BM}(X_{\et},\Q_{\ell}/\Z_{\ell}(i))}{H^{2i-1}_{\BMM}(X_{\et},\Q_{\ell}(i))}} \arrow[u,"\delta", "\cong"']
\end{tikzcd}
 \end{center}
commutes, where the isomorphism $\delta$ is induced from the Bockstein map (see Prop. \ref{prop:twisted-borel-moore-axioms} \ref{P3}), $\lambda^{i}_{X}$ is the map from \eqref{def:Blochs-map-singular} and the maps $\tilde{\alpha}^{i}_{X}$ and $\beta^{i}_{X}$ are given by \eqref{def:kahn-map} and \eqref{def:beta_X}, respectively. \end{corollary}
 \begin{proof} It follows immediately from Lemma \ref{lemma:beta_X=alpha_X} and Corollary \ref{cor:beta_X-vs-lambda_X}.\end{proof}
 The following is Theorem \ref{thm:lambda=alpha}, written in a slightly different form: If $X$ is a smooth variety over $k$, then the cycle class map $\beta^{i}_{X}$ from \eqref{def:beta_X} stated below, coincides with the map \eqref{def:etale-motivic-torsion-map} given in the introduction by Corollary \ref{cor:etale-motivic-cycle-class-map} and Lemma \ref{lemma:beta_X=alpha_X}. 
 \begin{corollary}[=Theorem \ref{thm:lambda=alpha}]\label{end} Let $X$ be a smooth and equi-dimensional algebraic $k$-scheme and let $\ell$ be a prime invertible in $k$. The following diagram \begin{center}
     \begin{tikzcd}
{\CH^{i}(X)[\ell^{\infty}]} \arrow[r, "\alpha^{i}_{X,\tors}"] \arrow[d, "-\lambda^{i}_{X}"'] \arrow[rd, "\beta^{i}_{X}"] & {H^{2i}_{L}(X,\Z_{\ell}(i))_{\tors}}                                                                 \\
{\frac{H^{2i-1}(X_{\et},\Q_{\ell}/\Z_{\ell}(i))}{M^{2i-1}(X)}} \arrow[r, equals]                      & {\frac{H^{2i-1}(X_{\et},\Q_{\ell}/\Z_{\ell}(i))}{H^{2i-1}_{M}(X,\Q_{\ell}(i))}} \arrow[u,"\delta", "\cong"'],
\end{tikzcd}
 \end{center}
commutes, where the isomorphism $\delta$ is induced from the Bockstein map (see Prop. \ref{prop:twisted-borel-moore-axioms} \ref{P3}), $\lambda^{i}_{X}$ is the map from \eqref{def:blochs-map} and the maps $\alpha^{i}_{X,\tors}$ and $\beta^{i}_{X}$ are given by \eqref{eq:alpha_{X}} and \eqref{def:beta_X}, respectively. 
 \end{corollary}
 \begin{proof} This readily follows from Corollary \ref{cor:l=a}. Since X is smooth and equi-dimensional, we find from Corollary \ref{cor:etale-motivic-cycle-class-map} that $\tilde{\alpha}^{i}_{X}=\alpha^{i}_{X}$. Moreover, the codomain of $\beta^{i}_{X}$ from \eqref{def:beta_X} is consistent with the one appearing in the statement of the Corollary because of the canonical identifications \eqref{etale=motivic} and \eqref{etale-vs-motivic}.\end{proof}
 
 \section*{Acknowledgements} I am grateful to Stefan Schreieder for his helpful comments concerning this work. This project has received funding from the European Research Council (ERC) under the European Union's Horizon 2020 research and innovation programme under grant agreement No 948066 (ERC-StG RationAlgic).\\

%\begin{thebibliography}{9}  

\end{document}